\documentclass{amsart}%[twoside]{article}
\usepackage{url,intmacros,graphicx}

\def\BibTeX{{\rm B\kern-.05em{\sc i\kern-.025em b}\kern-.08em
    T\kern-.1667em\lower.7ex\hbox{E}\kern-.125emX}}

\def\qed{\unskip\kern10pt{\unitlength1pt\linethickness{.4pt}\framebox(6,6){}}}

\newtheorem{theorem}{Theorem}[section]
\newtheorem{note}{Note}[section]
\newtheorem{example}{Example}[section]
\newcommand{\Mid}{\text{\rm m}}

\newcommand{\Rset}{\mathbb{R}}

\def\minus{%
  \setbox0=\hbox{-}%
  \vcenter{%
    \hrule width\wd0 height \the\fontdimen8\textfont3%
  }%
}

\usepackage{tikz}
\usetikzlibrary{arrows,trees}

\usepackage{amsmath}
\usepackage{amssymb}
\usepackage[]{algorithm}
\usepackage{algorithmic}
\usepackage{subfigure}
\usepackage{siunitx}

\newtheorem{definition}{Definition}[section]
\newtheorem{remark}{Remark}[section]

\begin{document}

% Fill in your title here. (Retain the footnote.)
\title{Runge--Kutta Theory and Constraint Programming}
\footnote{This is a revised version of ``Runge--Kutta Theory and Constraint Programming'', Reliable Computing vol. 25, 2017.}
%version of the paper published in the volume 25 of Reliable Computing in 2017 benefited of the remarks and corrections done by John C. Butcher.}}

% Delete the "\and" or add more as needed
\author{Julien Alexandre dit Sandretto}
% \addtocounter{footnote}{2}\footnote{Partially funded by the Academic and Research Chair:
% ``Complex Systems Engineering''- Ecole polytechnique ${\raise.17ex\hbox{$\scriptstyle\sim$}}$ THALES ${\raise.17ex\hbox{$\scriptstyle\sim$}}$ FX ${\raise.17ex\hbox{$\scriptstyle\sim$}}$ DGA ${\raise.17ex\hbox{$\scriptstyle\sim$}}$ DASSAULT AVIATION ${\raise.17ex\hbox{$\scriptstyle\sim$}}$ DCNS Research ${\raise.17ex\hbox{$\scriptstyle\sim$}}$
% ENSTA ParisTech ${\raise.17ex\hbox{$\scriptstyle\sim$}}$ Telecom ParisTech ${\raise.17ex\hbox{$\scriptstyle\sim$}}$ Fondation ParisTech ${\raise.17ex\hbox{$\scriptstyle\sim$}}$ FDO ENSTA}
\email{alexandre@ensta.fr}
\address{U2IS, ENSTA ParisTech, Universit\'e Paris-Saclay, 828 bd des Mar\'echaux, 91762 Palaiseau cedex France}
% \curraddr{U2IS, ENSTA ParisTech, Universit\'e Paris-Saclay, 828 bd des Mar\'echaux, 91762 Palaiseau cedex France}

 %\and {\bf Reviewed and corrected by John C. Butcher} }

% Put a short running title within the first argument to
% this command.  Do not aLTEr the second argument
% \markboth{\textit{Defining New Runge--Kutta methods with Interval Coefficients}}
%          {\textit{Revised version March 2018%Version reviewed and corrected by John C. Butcher 2017%Reliable Computing 25, 2017
%          }}

\date{March 2018}

\begin{abstract}
  There exist many Runge--Kutta methods (explicit or implicit), more or
  less adapted to specific problems.  Some of them have
  interesting properties, such as stability for stiff problems or
  symplectic capability for problems with energy conservation.  Defining
  a new method suitable to a given problem has become a
  challenge. The size, the complexity and the order do not stop 
  growing. 
  %This race
  This informal challenge to implement the best method is interesting but an
  important unsolved problem persists. Indeed, the coefficients of Runge--Kutta 
  methods are  harder and harder to compute, and the result is often expressed in
  floating-point numbers, which may lead to erroneous integration
  schemes. Here, we propose  to use interval analysis tools
  to compute Runge--Kutta coefficients. In particular, we use a solver based
  on guaranteed constraint programming.  Moreover, with a global
  optimization process and a well chosen cost function, we propose a
  way to define some novel optimal Runge--Kutta methods.
\end{abstract}
% Put keywords appropriate to your paper here, as shown

% Put your AMS subject classifications into the argument of
% the following command.
% \AMSsubj{34A45,65G20,65G40}

\maketitle

\section{Introduction}
\label{sec:introduction}

Many scientific applications in physical fields such as mechanics,
robotics, chemistry or electronics require solving differential
equations. This kind of equation appears for example, when the location is 
required, but only the velocity
and/or the acceleration are available when modelling a system. In
the general case, these differential equations cannot be formally
integrated, that is to say, closed form solutions are not available, and a
numerical integration scheme is used to approximate the state of the
system. The most classical approach is to use a Runge--Kutta scheme --
carefully chosen with respect to the problem, desired accuracy, and so
on -- to simulate the system behaviour.

Historically, the first method for numerical solution of differential equations was proposed by 
Euler in {\it
  Institutiones Calculi Integralis}~\cite{euler_1792}. His main idea
is based on a simple principle: if a particle is located at $y_0$ at
time $t_0$ and if its velocity at this time is known to be equal to
$v_0$, then at time $t_1$ the particle will be approximately at
position $y_1=y_0+(t_1-t_0) v_0$, under the condition that $t_1$ is
sufficiently close to $t_0$ (that is, after a very short time), so
velocity do not change ``too much'' over $[t_0,t_1]$.  Based on this principle,
around 1900 C. Runge and M. W. Kutta developed a family of iterative
methods, now called Runge--Kutta methods. While many such methods have been proposed since then, a unified
formalism and a deep analysis was first proposed by John Butcher in the
sixties~\cite{Butcher63}.

Almost from the beginning, after Euler, a race started to obtain new
schemes, with better properties or higher order of accuracy. It quickly became a
global competition. Recently, an explicit $14^{\mbox{\scriptsize{}th}}$ order Runge--Kutta scheme
with $35$ stages~\cite{rk14} and an implicit $17^{\mbox{\scriptsize{}th}}$ order Radau
with $9$ stages~\cite{radau17} were proposed. From
the beginning, methods have been discovered with the help of ingenuity in 
order to solve the highly complex problem, such as use of polynomials with
known zeros (Legendre for Gauss methods or Jacobi for
Radau)~\cite{HNW93}, vanishing of some coefficients~\cite{HNW93}, or
symmetry~\cite{rk14}. All these approaches, based on algebraic manipulations, are reaching their limit, 
due to the large number of stages. 
Indeed, to obtain a new method, we now need to solve a high-dimensional under-determined problem
with floating-point arithmetic~\cite{handbookfloat09}. Even if, as in some, multi-precision arithmetic
is used, the result obtained is still
not exact. A restriction  Runge--Kutta methods which have 
coefficients represented exactly in the computer can be eventually 
considered \cite{marciniak2004representation}. However, this restriction 
is really strong, because only few methods can be used, and it is the 
opposite of our approach. 

For this reason, in this paper we introduce application of interval coefficients for
Runge--Kutta methods; this could be an
interesting research direction for defining new reliable numerical
integration methods. 
We show that the properties of a Runge--Kutta scheme (such as order, stability, symplecticity, etc.) can be preserved with interval coefficients, 
while they are lost with floating-point numbers. 
By the use
of interval analysis tools~\cite{JKDW01,moore66}, and more
specifically a constraint programming (CP)
solver~\cite{rueher_csp_2005}, a general method to build
new methods with interval coefficients is presented. Moreover, an
optimization procedure allows us, with a well chosen cost function,
to define the optimal scheme.
The new methods with interval coefficients, obtained with our approach, have properties inclusion properties, meaning that the resulting interval box is guaranteed to contain 
a scheme that satisfies all the desired properties. 
They can be either used in a classical numerical integration procedure (but computations have to be done with interval arithmetic), or in a validated integration one \cite{alexandre2016validated}. In both cases, the properties of the scheme will be preserved. 

In this paper, a recurring reference will be made to the books by 
Hairer et al~\cite{HNW93}, which contains the majority of the results on
Runge--Kutta theory.

\paragraph{Outline.} We review the classical algorithm of a
simulation of an ordinary differential equation with Runge--Kutta
methods,  as well as a brief
introduction to the modern theory of Runge--Kutta methods, 
in Section~\ref{sec:recall-on-rk-methods}. In
Section~\ref{sec:rk_interval}, we present the interval analysis
framework used in this work and the advantages of having Runge--Kutta
methods with interval coefficients. We analyze some of the properties
of Runge--Kutta methods with and without interval coefficients in
Section~\ref{sec:properties}. In Section~\ref{sec:cp_cost}, the
constraint satisfaction problem to solve to obtain a new
scheme is presented.  In Section~\ref{sec:expe}, we present some
experimental results, followed in 
Section~\ref{sec:application} by the application  of the new schemes
in validated simulation. In
Section~\ref{sec:conclusion}, we summarize the main contributions of
the paper.

\paragraph{Notation.}
\begin{itemize}
 \item $\dot{y}$ denotes the time derivative of $y$, that is, $\frac{d y}{d
  t}$. 
 \item $a$ denotes a real value, while $\textbf{a}$ represents a
vector of real values. 
\item $[a]$ represents an interval value and $[\mathbf{a}]$ represents a vector of interval values (a box). 
\item The midpoint of an interval $[x]$ is denoted by $\Mid([x])$. 
\item The variables $y$ are used for the state variables of the system and $t$
represents time. 
\item Sets will be represented by calligraphic letter such as $\mathcal{X}$ or $\mathcal{Y}$. 
\item The real part and the imaginary part of a complex number $z$ will be denoted by $Re(z)$ and
$Im(z)$ respectively.
\item An interval with floating point bounds is written in the short form, for example\\ $0.123456[7,8]$ to represent the interval $[0.1234567,0.1234568]$. 
\end{itemize}

\section{A Review of Runge--Kutta Methods}
\label{sec:recall-on-rk-methods}

Historically, Runge--Kutta methods were used to compute a Taylor
series expansion without any derivative computation, which was a
difficult problem in the $19^{\mbox{\scriptsize{}th}}$ Century. Now, \emph{automatic
  differentiation} methods~\cite{Griewank:SIAM00} can be used to
efficiently compute derivatives, but Runge--Kutta methods are more than
a simple technique to compute a Taylor series expansion. Mainly,
Runge--Kutta methods have strong stability properties (see
Section~\ref{sec:properties} for a more formal definition), which make
them suitable for efficiently solving different classes of problems,
especially stiff systems. In particular, implicit methods can be
algebraically stable, stiffly accurate and symplectic (see Section~\ref{symplecticity}). For this
reason, the study of the properties of Runge--Kutta methods is highly
interesting, and the definition of new techniques to build new
Runge--Kutta methods with strong properties is also of interest.

\subsection{Numerical Integration with Runge--Kutta Methods}
\label{sec:numer-solut-init}

Runge--Kutta methods can solve the \emph{initial value problem}
(\textit{IVP}) of non-autonomous \emph{Ordinary Differential
  Equations} (ODEs) defined by
\begin{equation}
  \label{eq:ivp}
  \dot{\mathbf{y}}=\mathbf{f}(t, \mathbf{y})\quad\text{with}\quad
  \mathbf{y}(0) = \mathbf{y}_{0}\quad\text{and}\quad t\in[0,t_{\text{end}}]
  \enspace.
\end{equation}
The function $\mathbf{f}: \Rset \times \Rset^n \to \Rset^n$ is called the
{\it vector field}, $\mathbf{y}\in\Rset^n$ is called the {\it vector of state variables}, and
$\dot{\mathbf{y}}$ denotes the derivative of $\mathbf{y}$ with respect to
time $t$. We shall always assume at least that $\mathbf{f}$ is
globally Lipschitz in $\mathbf{y}$, so Equation~\eqref{eq:ivp} admits
a unique solution~\cite{HNW93} for a given initial condition
$\mathbf{y}_0$. Furthermore, for our purpose, we shall assume,  as needed, that
$\mathbf{f}$ is continuously differentiable. The exact
solution of Equation~\eqref{eq:ivp} is denoted by $\mathbf{y}(t;
\mathbf{y}_0)$, the {\it flow}.

The goal of a numerical simulation to solve Equation~\eqref{eq:ivp} is
to compute a sequence of time instants $0 = t_{0} < t_{1} < \cdots <
t_{N} = t_{\text{end}}$ (not necessarily equidistant) and a sequence
of states $\mathbf{y}_0,\ \dots,\ \mathbf{y}_N$ such that $\forall
\ell\in[0,N]$, $\mathbf{y}_{\ell} \approx
\mathbf{y}(t_\ell,\mathbf{y}_{\ell-1})$, obtained with the help of an
integration scheme.

A Runge--Kutta method, starting from an initial value $\mathbf{y}_\ell$
at time $t_\ell$ and a finite time horizon $h$, the \emph{step-size},
produces an approximation $\mathbf{y}_{\ell+1}$ at time $t_{\ell+1}$,
with $t_{\ell+1}-t_\ell = h$, of the solution $\mathbf{y}(t_{\ell+1};
\mathbf{y}_\ell)$. Furthermore, to compute $\mathbf{y}_{\ell+1}$, a
Runge--Kutta method computes $s$ evaluations of $f$ at predetermined
time instants. The number $s$ is known as the number of \emph{stages}
of a Runge--Kutta method. More precisely, a Runge--Kutta method is
defined by
\begin{equation}
  \label{eq:rk-increment}
  \mathbf{y}_{\ell+1} = \mathbf{y}_\ell + h \sum_{i=1}^{s}b_i \mathbf{k}_i,
\end{equation}
with $\mathbf{k}_i$ defined by
\begin{equation}
  \label{eq:rk-intermediate}
  \mathbf{k}_i = \mathbf{f}\left(t_\ell + c_i h,
    \mathbf{y}_\ell + h \sum_{j=1}^{s} a_{ij}\mathbf{k}_j \right).
\end{equation}
The coefficients $c_i$, $a_{ij}$ and $b_i$, for $i,j=1,2,\cdots, s$,
fully characterize the Runge--Kutta methods, and they are usually
synthesized in a \textit{Butcher tableau} \cite{Butcher63} of the form
\begin{displaymath}
  \begin{aligned}
    \begin{array}{c|cccc}
      c_1 & a_{11} & a_{12} & \dots & a_{1s}\\
      c_2 & a_{21} & a_{22} & \dots & a_{2s}\\
      \vdots & \vdots & \vdots & \ddots & \vdots\\
      c_s & a_{s1} & a_{s2} & \dots & a_{ss} \\
      \hline & b_1 & b_2 & \dots & b_s\\
    \end{array}
    & \qquad\equiv\qquad &
    \begin{array}{c | c}
      \mathbf{c} & \mathbf{A}
      \\
      \hline
      & \mathbf{b}
    \end{array}
  \end{aligned}
  \enspace.
\end{displaymath}

In terms of the form of the matrix $\mathbf{A}$, consisting of the
coefficients $a_{ij}$, a Runge--Kutta method can be
\begin{itemize}
\item \emph{explicit}, for example, as in the classical Runge--Kutta method of
  order~$4$ given in Figure~\ref{fig:butcher-tableau-rk4}. In other
  words, the computation of the intermediate $\mathbf{k}_i$ only
  depends on the previous steps $\mathbf{k}_j$ for $j < i$;
\item \emph{diagonally implicit}, for example, as in the diagonally implicit
  fourth-order method given in
  Figure~\ref{fig:butcher-tableau-sdirk4}. In this case, the
  computation of an intermediate step $\mathbf{k}_i$ involves the
  value $\mathbf{k}_i$, so non-linear systems in $\mathbf{k}_i$
  must be solved. A method is \emph{singly diagonally implicit} if the
  coefficients on the diagonal are all equal;
\item \emph{fully implicit}, for example, the Runge--Kutta fourth-order
  method with a Lobatto quadrature formula given in
  Figure~\ref{fig:butcher-tableau-lobatto3c4}. In this last case, the
  computation of intermediate steps involves the solution of a
  non-linear system of equations in all the values $\mathbf{k}_i$ for
  $i=1,2,\cdots,s$.
\end{itemize}

\begin{figure}[t]
  \footnotesize
  \centering
  \subfigure[RK4]{
    \label{fig:butcher-tableau-rk4}
    \begin{tabular}{r|rrrr}
     $0$ & $0$ & $0$ & $0$ & $0$\\
        $\frac{\strut 1}{2}$ & $\frac{\strut 1}{2}$ & $0$ & $0$ & $0$\\
        $\frac{\strut 1}{2}$ & $0$ & $\frac{\strut 1}{2}$ & $0$ & $0$\\
        $1$ & $0$ & $0$ & $1$ & $0$\\[3pt]
        \hline & $\frac{\strut 1}{6}$ & $\frac{\strut 1}{3}$ & $\frac{\strut 1}{3}$ & $\frac{\strut 1}{6}$
    \end{tabular}

%     \begin{math}
%       \begin{array}{c|cccc}
%         0 & 0 & 0 & 0 & 0\\
%         \frac{\strut 1}{2} & \frac{1}{2} & 0 & 0 & 0\\
%         \frac{\strut 1}{2} & 0 & \frac{1}{2} & 0 & 0\\
%         1 & 0 & 0 & 1 & 0\\[3pt]
%         \hline & \frac{\strut 1}{6} & \frac{1}{3} & \frac{1}{3} & \frac{1}{6}
%       \end{array}
%     \end{math}
  }
  \quad
    \subfigure[Lobatto3c]{
    \label{fig:butcher-tableau-lobatto3c4}
%     \begin{math}
%       \begin{array}{c|ccc}
%         0 & \frac{\strut 1}{6} & -\frac{1}{3}& \frac{1}{6} \\
%         \frac{\strut 1}{2} & \frac{1}{6} & \frac{5}{12}& -\frac{1}{12}\\
%         1 & \frac{\strut 1}{6} & \frac{2}{3} & \frac{1}{6} \\[3pt]
%         \hline
%         & \frac{\strut 1}{6} & \frac{\strut 2}{3} & \frac{\strut 1}{6}
%       \end{array}
%     \end{math}
    \begin{tabular}{r|rrr}
        $0$ & $\frac{\strut 1}{6}$ & $\minus\frac{\strut 1}{3}$& $\frac{\strut 1}{6}$ \\
        $\frac{\strut 1}{2}$ & $\frac{\strut 1}{6}$ & $\frac{\strut 5}{12}$& $\minus\frac{\strut 1}{12}$\\
        $1$ & $\frac{\strut 1}{6}$ & $\frac{\strut 2}{3}$ & $\frac{\strut 1}{6}$ \\[5pt]
        \hline
        & $\frac{\strut 1}{6}$ & $\frac{\strut 2}{3}$ & $\frac{\strut 1}{6}$
    \end{tabular}

  }
  
  \subfigure[SDIRK4]{
    \label{fig:butcher-tableau-sdirk4}
    %\begin{math}
      %\begin{array}{c|ccccc}
%         \frac{\strut 1}{4} &\frac{1}{4} & & & & \\%0 & 0 & 0 & 0 \\
%         \frac{\strut 3}{4} &\frac{1}{2} & \frac{1}{4} & & & \\%0 & 0 & 0\\
%         \frac{\strut 11}{20} &\frac{17}{50} & \frac{-1}{25} & \frac{1}{4} & & \\%0 & 0\\
%         \frac{\strut 1}{2} &\frac{371}{1360} & \frac{-137}{2720} & \frac{15}{544} & \frac{1}{4} & \\%0\\
%         1 &\frac{\strut 25}{24} & \frac{-49}{48} & \frac{125}{16} & \frac{-85}{12} & \frac{1}{4}\\[3pt]
%         \hline & \frac{\strut 25}{24} & \frac{-49}{48} & \frac{125}{16} & \frac{-85}{12} & \frac{1}{4}
      %\end{array}
    %\end{math}
    \begin{tabular}{r|rrrrr}
     $\frac{\strut 1}{4}$ &$\frac{\strut 1}{4}$ & & & & \\%0 & 0 & 0 & 0 \\
       $ \frac{\strut 3}{4}$ &$\frac{\strut 1}{2}$ & $\frac{\strut 1}{4}$ & & & \\%0 & 0 & 0\\
       $ \frac{\strut 11}{20}$ &$\frac{\strut 17}{50}$ & $\minus \frac{\strut 1}{25}$ & $\frac{\strut 1}{4}$ & & \\%0 & 0\\
       $ \frac{\strut 1}{2} $&$\frac{\strut 371}{1360}$ & $\minus\frac{\strut 137}{2720}$ & $\frac{\strut 15}{544}$ & $\frac{\strut 1}{4}$ & \\%0\\
       $ 1$ &$\frac{\strut 25}{24} $& $\minus\frac{\strut 49}{48}$ & $\frac{\strut 125}{16}$ & $\minus\frac{\strut 85}{12}$ & $\frac{\strut 1}{4}$\\[3pt]
        \hline & $\frac{\strut 25}{24}$ & $\minus\frac{\strut 49}{48}$ & $\frac{\strut 125}{16}$ & $\minus\frac{\strut 85}{12}$ & $\frac{\strut 1}{4}$
    \end{tabular}

  }

  \caption{Different kinds of Runge--Kutta methods}
  \label{fig:different-kinds-rk}
\end{figure}

The order of a Runge--Kutta method is $p$ if and only if the
\emph{local truncation error}, in other words, the distance between the exact
solution $\mathbf{y}(t_\ell; \mathbf{y}_{\ell-1})$ and the numerical solution $\mathbf{y}_{\ell}$
is such that:
\begin{displaymath}
  \mathbf{y}(t_\ell; \mathbf{y}_{\ell-1}) - \mathbf{y}_{\ell} = \mathcal{O}(h^{p+1})\enspace.
\end{displaymath}

Some theoretical results have been obtained concerning the relation between 
the number of stages $s$ and the order $p$. For the explicit methods, 
there  is no Runge--Kutta method of order $p$ with $s=p$ stages when $p > 4$. 
For the implicit methods, $p=2s$ is the largest possible order for a given 
number of stages, and only Gauss-Legendre methods have this capability \cite{HNW93}. 

\subsection{Butcher's Theory of Runge--Kutta Methods}
\label{sec:butcher-theory-Runge--Kutta}

One of the main ideas of Butcher in~\cite{Butcher63} is to
express the
Taylor expansion of the exact solution of~\eqref{eq:ivp} and the Taylor expansion of
the numerical solution using the same basis of \emph{elementary differentials} . 
The elementary differentials are made of sums
of partial derivatives of $f$ with respect to the components of
$\mathbf{y}$. Another salient idea of Butcher in~\cite{Butcher63}
is to relate these partial derivatives of order $q$ to a combinatorial
problem to enumerate all the trees $\tau$ with exactly $q$ nodes. From
the structure of a tree $\tau$, one can map a particular partial
derivative; see
Table~\ref{tab:trees-elementary-differentials-coefficients} for some
examples. It follows that one has the three following theorems, used
used to express the order condition of Runge--Kutta methods. In
theorems~\ref{thm:derivative-exact-solution}
and~\ref{thm:derivative-numerical-solution}, $\tau$ is a rooted tree,
$F(\tau)$ is the elementary differential associated with $\tau$,
$r(\tau)$ is the order of $\tau$ (the number of nodes it contains),
$\gamma(\tau)$ is the density, $\alpha(\tau)$ is the number of
equivalent trees and $\psi(\tau)$ the elementary weight of $\tau$
based on the coefficients $c_i$, $a_{ij}$ and $b_i$ defining a
Runge--Kutta method; see \cite{Butcher63} for more
details. Theorem~\ref{thm:derivative-exact-solution} defines the
$q$-th time derivative of the exact solution expressed with elementary
differentials. Theorem~\ref{thm:derivative-numerical-solution} defines
the $q$-th time derivative of the numerical solution expressed with elementary
differentials. Finally, Theorem~\ref{thm:lte-butcher} formally defines
the order condition of the Runge--Kutta methods.

\begin{table}[!t]
  \centering
  \caption{Rooted trees $\tau$, elementary differentials $F(\tau)$, and
    their coefficients}
  \label{tab:trees-elementary-differentials-coefficients}
  \begin{tabular}{| c | c | c | c | c | c |}
    \hline
    $r(\tau)$ & Trees & $F(\tau)$ &  $\alpha(\tau)$ & $\gamma(\tau)$ &
    $\phi(\tau)$
    \\
    \hline
    \hline
    %%% order 1
    $1$ &
    \begin{tikzpicture}
      \node[draw, circle, inner sep=1pt, fill] {};
    \end{tikzpicture}
    &
    $\mathbf{f}$
    &
    $1$
    &
    $1$
    &
    $\sum_i b_i$
    \\
    \hline
    %%% order 2
    $2$
    &
    \begin{tikzpicture}[node distance=10pt]
      \node (up) [draw, circle, inner sep=1pt, fill] {};
      \node (down) [draw, circle, inner sep=1pt, fill, below of=up] {};
      \draw[-] (up) -- (down);
    \end{tikzpicture}
    &
    $\mathbf{f}'\mathbf{f}$
    &
    $1$
    &
    $2$
    &
    $\sum_{ij}b_ia_{ij}$
    \\
    %%% order 3
    \hline
    $3$
    &
    \begin{tikzpicture}[node distance=10pt]
      \node (r1) [draw, circle, inner sep=1pt, fill] {};
      \node (t1) [draw, circle, inner sep=1pt, fill, above left of=r1] {};
      \node (t2) [draw, circle, inner sep=1pt, fill, above right of=r1] {};
      \draw[-] (r1) -- (t1);
      \draw[-] (r1) -- (t2);
    \end{tikzpicture}
    &
    $\mathbf{f}''(\mathbf{f}, \mathbf{f})$
    &
    $1$
    &
    $3$
    &
    $\sum_{ijk}b_i a_{ij}a_{ik}$
    \\
    $3$
    &
    \begin{tikzpicture}[node distance=10pt]
      \node (r1) [draw, circle, inner sep=1pt, fill] {};
      \node (t1) [draw, circle, inner sep=1pt, fill, above left of=r1] {};
      \node (t2) [draw, circle, inner sep=1pt, fill, above right of=t1] {};
      \draw[-] (r1) -- (t1);
      \draw[-] (t1) -- (t2);
    \end{tikzpicture}
    &
    $\mathbf{f}'\mathbf{f}'\mathbf{f}$
    &
    $1$
    &
    $6$
    &
    $\sum_{ijk} b_i a_{ij}a_{jk}$
    \\
    %%% order 4
    \hline
    $4$
    &
    \begin{tikzpicture}[node distance=10pt]
      \node (r1) [draw, circle, inner sep=1pt, fill] {};
      \node (t1) [draw, circle, inner sep=1pt, fill, above left of=r1] {};
      \node (t2) [draw, circle, inner sep=1pt, fill, above right of=r1] {};
      \node (t3) [draw, circle, inner sep=1pt, fill, above of=r1] {};
      \draw[-] (r1) -- (t1);
      \draw[-] (r1) -- (t2);
      \draw[-] (r1) -- (t3);
    \end{tikzpicture}
    &
    $\mathbf{f}'''(\mathbf{f}, \mathbf{f}, \mathbf{f})$
    &
    $1$
    &
    $4$
    &
    $\sum_{ijkl} b_i a_{ij}a_{ik}a_{il}$
    \\
    $4$
    &
    \begin{tikzpicture}[node distance=10pt]
      \node (r1) [draw, circle, inner sep=1pt, fill] {};
      \node (t1) [draw, circle, inner sep=1pt, fill, above left of=r1] {};
      \node (t2) [draw, circle, inner sep=1pt, fill, above right of=r1] {};
      \node (t3) [draw, circle, inner sep=1pt, fill, above left of=t2] {};
      \draw[-] (r1) -- (t1);
      \draw[-] (r1) -- (t2);
      \draw[-] (t2) -- (t3);
    \end{tikzpicture}
    &
    $\mathbf{f}''(\mathbf{f}'\mathbf{f}, \mathbf{f})$
    &
    $3$
    &
    $8$
    &
    $\sum_{ijkl} b_i a_{ij}a_{ik}a_{jl}$
    \\
    $4$
    &
    \begin{tikzpicture}[node distance=10pt]
      \node (r0) [draw, circle, inner sep=1pt, fill] {};
      \node (r1) [draw, circle, inner sep=1pt, fill, above of=r0] {};
      \node (t1) [draw, circle, inner sep=1pt, fill, above left of=r1] {};
      \node (t2) [draw, circle, inner sep=1pt, fill, above right of=r1] {};
      \draw[-] (r0) -- (r1);
      \draw[-] (r1) -- (t1);
      \draw[-] (r1) -- (t2);
    \end{tikzpicture}
    &
    $\mathbf{f}'\mathbf{f}''(\mathbf{f},\mathbf{f})$
    &
    $1$
    &
    $12$
    &
    $\sum_{ijkl} b_i a_{ij}a_{jk}a_{jl}$
    \\
    $4$
    &
    \begin{tikzpicture}[node distance=10pt]
      \node (r1) [draw, circle, inner sep=1pt, fill] {};
      \node (t1) [draw, circle, inner sep=1pt, fill, above left of=r1] {};
      \node (t2) [draw, circle, inner sep=1pt, fill, above right of=t1] {};
      \node (t3) [draw, circle, inner sep=1pt, fill, above left of=t2] {};
      \draw[-] (r1) -- (t1);
      \draw[-] (t1) -- (t2);
      \draw[-] (t2) -- (t3);
    \end{tikzpicture}
    &
    $\mathbf{f}'\mathbf{f}'\mathbf{f}'\mathbf{f}$
    &
    $1$
    &
    $24$
    &
    $\sum_{ijkl} b_i a_{ij}a_{jk}a_{kl}$
    \\
    \hline
  \end{tabular}
\end{table}

\begin{theorem}
  \label{thm:derivative-exact-solution}
  The $q$-th derivative w.r.t. time of the \emph{exact solution} is
  given by
  \begin{displaymath}
    \mathbf{y}^{(q)} = \sum_{r(\tau)=q} \alpha(\tau) F(\tau)(\mathbf{y}_0)
    \enspace.
  \end{displaymath}
\end{theorem}

\begin{theorem}
  \label{thm:derivative-numerical-solution}
  The $q$-th derivative w.r.t. time of the \emph{numerical solution}
  is given by
  \begin{displaymath}
    \mathbf{y}_1^{(q)}=\sum_{r(\tau)=q} \gamma(\tau) \phi(\tau) \alpha(\tau)
    F(\tau)(\mathbf{y}_0)
    \enspace.
  \end{displaymath}
\end{theorem}

\begin{theorem}[Order condition]
  \label{thm:lte-butcher}
  A Runge--Kutta method has order $p$ iff
  \begin{displaymath}
    \phi(\tau) = \frac{1}{\gamma(\tau)}\quad
    \forall \tau, r(\tau) \leqslant p
    \enspace.
  \end{displaymath}
\end{theorem}

These theorems give the necessary and sufficient conditions to define
new Runge--Kutta methods. In other words, they define a system of
equations, where the unkowns are the coefficients $c_i$, $a_{ij}$ and
$b_i$, which characterize a Runge--Kutta method. For example, for the
first four orders, and following the order condition, the following
constraints on the derivative order have to be solved to create a new
Runge--Kutta method
\begin{itemize}
\item order $1$: $\sum_i b_i = 1$
\item order $2$: $\sum_{i} b_i c_i = \frac{1}{2}$
\item order $3$: $\sum_{ij} b_i a_{ij}c_j  = \frac{1}{6}$ ,
  $\sum_i b_i c_i^2 = \frac{1}{3}$
\item order $4$: $\sum_i b_i c_i^3=\frac{1}{4}$, $\sum_{ij} b_i c_i a_{ij}c_j=\frac{1}{8}$, $\sum_{ij} b_i a_{ij} c_j^2=\frac{1}{12}$, $\sum_{ijk}  b_i a_{ij} a_{jk} c_k=\frac{1}{24}$
\end{itemize}
The total number of constraints increases rapidly: $8$ for the
$4^{\mbox{\scriptsize{}th}}$ order, $17$ for the $5^{\mbox{\scriptsize{}th}}$ order, $37$, $85$, $200$, etc.  Note
also an additional constraint, saying that the $c_i$ must be
increasing, has to be taken into account, and also that $c_i$ are such
that
\begin{displaymath}
  c_i = \sum_j a_{ij}\enspace.
\end{displaymath}

\begin{note}
Butcher noticed that the constraint such that the $c_i$ have to increase is not true. Indeed, it is not true for the method given in Figure~\ref{fig:butcher-tableau-sdirk4}. 
This constraint can then be relaxed. 
\end{note}

These constraints are the smallest set of constraints, known as
\emph{Butcher rules}, which have to be validated in order to define new
Runge--Kutta methods.

Additionally, other constraints can be added to define particular
structure of Runge--Kutta methods~\cite{Butcher63}, as for example, to make it
\begin{itemize}
\item Explicit: $a_{ij} = 0, \forall j \ge i$
\item Singly diagonal: $ a_{1,1} = \dots = a_{s,s}$
\item Diagonal implicit: $a_{ij} = 0, \forall j>i$
\item Explicit first line: $a_{11}=\dots=a_{1s}=0$
\item Stiffly accurate: $a_{si} = b_i, \forall i=1,\dots,s$
% \item Fully implicit: $a_{ij} \neq 0, \forall i,j=1,\dots,s$
\end{itemize}

Note that historically, some simplifications of this set of constraints
were used to reduce the complexity of the problem. For example, to
obtain a fully implicit scheme with a method based on Gaussian
quadrature (see \cite{butcher2003numerical} for more details), the $c_1,\dots,c_s$ are the zeros of the shifted Legendre
polynomial of degree $s$, given by:
\begin{equation*}
  \frac{d^s}{dx^s}(x^s(x-1)^s).
\end{equation*}
This approach is called the ``Kuntzmann-Butcher methods'' and is used
to characterize the Gauss-Legendre methods~\cite{butcher2003numerical}. Another example: by finding the zeros of
\begin{equation*}
 \frac{d^{s-2}}{dx^{s-2}}(x^{s-1}(x-1)^{s-1}),
\end{equation*}
the Lobatto quadrature formulas are obtained (see Figure~\ref{fig:butcher-tableau-lobatto3c4}). When the zeros of 
\begin{equation*}
  \frac{d^{s-1}}{dx^{s-1}}(x^{s-1}(x-1)^s).
\end{equation*}
provide the famous Radau IIA quadrature formulas. 

The problems with this approach are obvious. First, the resulting Butcher tableau is guided by the solver and not by the requirements on the
properties. Second, a numerical computation in
floating-point numbers is needed, and because such computations are not exact, the constraints may not be satisfied. 

We propose an interval analysis approach to solve these constraints
and hence produce reliable results. More precisely, we follow the
\emph{constraint satisfaction problem} approach.

\section{Runge--Kutta  with Interval Coefficients}
\label{sec:rk_interval}

As seen before in Section~\ref{sec:butcher-theory-Runge--Kutta}, the
main constraints are the order conditions, also called
\emph{Butcher rules}. Two other constraints need to be considered: the
sum of $a_{ij}$ is equal to $c_i$ for all the table lines; and the
$c_i$ are increasing with respect to $i$. These constraints have to be
fulfilled to obtain a valid Runge--Kutta method, and they can be
gathered in a Constraint Satisfaction Problem (CSP).
\begin{definition}[CSP]
  A numerical (or continuous) CSP $(\mathcal{X},\mathcal{D},
  \mathcal{C})$ is defined as follows:
  \begin{itemize}
  \item $\mathcal{X} = \{x_1,\hdots,x_n\}$ is a set of variables,
    also represented by the vector $\mathbf{x}$.
  \item $\mathcal{D} = \{[x_1],\hdots,[x_n]\}$ is a set of domains
    ($[x_i]$ contains all possible values of $x_i$).
  \item $\mathcal{C} = \{c_1,\hdots,c_m\}$ is a set of constraints of
    the form $c_i(\mathbf{x}) \equiv f_i(\mathbf{x}) = 0$ or
    $c_i(\mathbf{x}) \equiv g_i(\mathbf{x}) \leqslant 0$, with $f_i :
    \Rset^n \to \Rset$, $g_i : \Rset^n \to \Rset$ for $1 \leqslant i
    \leqslant m$. Constraints $\mathcal{C}$ are interpreted as a
    conjunction of equalities and inequalities.
  \end{itemize}
\end{definition}
An evaluation of the variables is a function from a subset of variables to a set of values in the corresponding 
subset of domains. An evaluation is {\it consistent} if no constraint is violated. An evaluation is {\it complete} if it includes 
all variables. The {\it solution} of a CSP is a complete and consistent evaluation. %of $\mathbf{x}$ ranging in $[\mathbf{x}]$ and satisfying the constraints~$\mathcal{C}$.

In the particular case of continuous (or numerical) CSPs, interval based techniques provide generally one or a list of boxes which enclose the solution. 
The CSP approach is at the same time powerful enough to address complex
problems (NP-hard problems with numerical issues, even in critical
applications) and simple in the definition of a solving framework
\cite{Benhamou:1994:CR:864422,Lebbah2002109}.

Indeed, the classical algorithm to solve a CSP is the branch-and-prune algorithm, 
which needs only an evaluation of the constraints and an initial
domain for variables. While this algorithm is sufficient for many
problems, to solve other problems, some improvements have been achieved, and algorithms based on
contractors have emerged \cite{chabert2009contractor}. The branch-and-contract 
algorithm consists of two main steps: i) the contraction (or
filtering) of one variable and the propagation to the others until a
fixed point reached, then ii) the bisection of the domain of one
variable in order to obtain two problems, easier to solve.

A more detailed description
follows.

\paragraph{Contraction}
A filtering algorithm or contractor is used in a CSP solver to reduce
the domain of variables to a fixed point (or a near fixed point), by
respecting local consistencies. A contractor $Ctc$ can be defined
with the help of constraint programming, analysis or algebra, but it
must satisfy three properties:
\begin{itemize}
 \item $Ctc (\mathcal{D}) \subseteq \mathcal{D}$: contractivity,
 \item $Ctc$ cannot remove any solution: it is conservative,
 \item $\mathcal{D}' \subseteq \mathcal{D} \Rightarrow Ctc
   (\mathcal{D}') \subseteq Ctc(\mathcal{D})$: monotonicity.
\end{itemize}
There are many contractor operators defined in the literature, most notably:
\begin{itemize}
\item (Forward-Backward contractor) By considering only one
  constraint, this method computes the interval enclosure of a node in
  the tree of constraint operations with the children domains (the
  forward evaluation), then refines the enclosure of a node in terms
  of parents domain (the backward propagation). For example, from the
  constraint $x+y = z$, this contractor refines initial domains $[x]$,
  $[y]$ and $[z]$ from a forward evaluation $[z] = [z] \cap
  ([x]+[y])$, and from two backward evaluations $[x] = [x] \cap
  ([z]-[y])$ and $[y] = [y] \cap ([z]-[x])$.
\item (Newton contractor) This contractor, based on the first order Taylor
  interval extension: $[f]([\mathbf{x}]) = f(\mathbf{x}^*) +
  [J_f]([\mathbf{x}]) ([\mathbf{x}] - \mathbf{x}^*)$ with
  $\mathbf{x}^* \in [\mathbf{x}]$, has the property: if $0 \in
  [f]([\mathbf{x}])$, then $[\mathbf{x}]_{k+1} = [\mathbf{x}]_k \cap
  x^*-[J_f]([\mathbf{x}]_k)^{-1} f(\mathbf{x}^*)$ is a tighter
  inclusion of the solution of $f(x)=0$.  Some other contractors based
  on Newton's method, such as the Krawczyk operator~\cite{JKDW01}, have been defined.
\end{itemize}

\paragraph{Propagation}
If a variable domain has been reduced, the reduction is propagated to
all the constraints involving that variable, allowing the
other variable domains to be narrowed. This process is repeated until a fixed point is
reached.

\paragraph{Branch-and-Prune}
A Branch-and-Prune algorithm consists on alternatively br\-an\-ching and
pruning to produce two sub-pavings $\mathcal{L}$ and $\mathcal{S}$,
with $\mathcal{L}$ the boxes too small to be bisected and
$\mathcal{S}$ the solution boxes. We are then sure that all solutions
are included in $\mathcal{L} \cup \mathcal{S}$ and that every point in
$\mathcal{S}$ is a solution.

Specifically, this algorithm traverses a list of boxes $\mathcal{W}$,
initialized $\mathcal{W}$  with the vector $[\mathbf{x}]$ consisting of
the elements of $\mathcal{D}$.  For each box in  $\mathcal{W}$, the following
is done:  i) Prune: the CSP is
evaluated (or contracted) on the current box; if the box is is a solution, it
is added to $\mathcal{S}$; otherwise ii) Branch: if the box is large
enough, it is bisected and the two boxes resulting are added into
$\mathcal{W}$; otherwise the box is added to $\mathcal{L}$.
 
\begin{example}
  An example of the problems that the previously presented tools can solve
  is taken from \cite{Lhomme93consistencytechniques}. The 
  CSP is defined as follows:
  \begin{itemize}
  \item $\mathcal{X} = \{x,y,z,t\}$
  \item $\mathcal{D} = \{[x]=[0,1000],
    [y]=[0,1000],[z]=[0,3.1416],[t]=[0,3.1416]\}$
  \item $\mathcal{C} =
    \{xy+t-2z=4;x\sin(z)+y\cos(t)=0;x-y+\cos^2(z)=\sin^2(t);xyz=2t\}$
  \end{itemize}
  We use a Branch-and-Prune algorithm with the
  Forward-Back\-ward contractor and a propagation algorithm to solve this CSP. The
  solution ($[1.999, 2.001]$, $[1.999, 2.001]$, $[1.57, 1.571]$,
  $[3.14159, 3.1416]$) is obtained with only $6$ bisections.
  \hfill$\blacksquare$
\end{example}

\subsection{Correctness of CSP Applied to Butcher Rules}

By construction, the CSP approach guarantees that the exact
solution of the problem, denoted by $(\tilde{a}_{ij}, \tilde{b}_i,
\tilde{c}_i)$, is included in the solution provided by the
corresponding solver, given by $([a_{ij}], [b_i], [c_i])$. The Butcher
rules are then preserved by inclusion through the use of interval
coefficients.

\begin{theorem}
  If Runge--Kutta coefficients are given by intervals obtained by a CSP
  solver on constraints coming from the order condition defined in
  Theorem~\ref{thm:lte-butcher} then they contain at least one solution which satisfies the Butcher rules.
\end{theorem}

\begin{proof}
  Starting from the order condition defined in
  Theorem~\ref{thm:lte-butcher}, and given the additional details in
  \cite{alexandre2016validated}, if the Runge--Kutta coefficients are
  given by intervals, such that $\tilde{a}_{ij} \in [a_{ij}],
  \tilde{b}_i \in [b_i], \tilde{c}_i \in [c_i]$, then $[\phi(\tau)]
  \ni \frac{1}{\gamma(\tau)}\quad\forall \tau, r(\tau) \leq p$. In
  other words, $\mathbf{y}^{(q)} \in [\mathbf{y}_1^{(q)}], \forall q
  \leq p$, and then the derivatives of the exact solution are included
  in the numerical ones, and the Taylor series expansion of
  the exact solution is included (monotonicity of the interval sum) in
  the Taylor series expansion of the numerical solution obtained from the
  Runge--Kutta method with interval coefficients.
\end{proof}

\begin{remark}
  If a method is given with interval coefficients such that
  $\tilde{a}_{ij} \in [a_{ij}], \tilde{b}_i \in [b_i], \tilde{c}_i \in
  [c_i]$, there is an over-estimation of the derivatives
  $|\mathbf{y}^{(q)} - [\mathbf{y}_1^{(q)}]|$. To make this
  over-approximation as small as possible, the enclosure of the
  coefficients has to be as sharp as possible.
\end{remark}

\subsection{Link with Validated Numerical Integration Methods}

To make the Runge--Kutta method validated \cite{alexandre2016validated}, the
challenging question is how to compute a bound on the difference between
the true solution and the numerical solution, defined by
$\mathbf{y}(t_\ell;\mathbf{y}_{\ell-1}) - \mathbf{y}_\ell$. This
distance is associated with the \emph{local truncation error} (LTE) of
the numerical method. We showed that LTE can be easily bounded by
using the difference between the Taylor series of the exact and the
numerical solutions, which is reduced to $\text{LTE} =
\mathbf{y}^{(p+1)}(t_\ell) - [\mathbf{y}_\ell^{(p+1)}]$, with $p$ the order of
the method undere consideration. This difference has to be evaluated on a
specific box, obtained with the Picard-Lindel\"of operator, but this
is outside the scope of this paper, see \cite{alexandre2016validated}
for more details. For a method with interval coefficients, the LTE is
well bounded (even over-approximated), which is not the case for a
method with floating-point coefficients. For a validated method, the
use of interval coefficients is then a requirement.

\section{Stability Properties with Interval Coefficients}
\label{sec:properties}

Runge--Kutta methods have strong stability properties
which are not present for other numerical integration methods such as
multi-step methods, for example, Adams-Moulton methods or BDF
methods~\cite{HNW93}. It is interesting to understand that these
properties, proven in theory, are lost in practice if we use
floating-point number coefficients. In this section, we show 
that the properties of Runge--Kutta methods are preserved with the use
of interval coefficients in the Butcher tableau. The definition of
stability can have a very different form depending on the class of 
problems under consideration.

\subsection{Notion of Stability}
\label{sec:notion-stability}

In \cite{HNW93}, the authors explain that when we do not have the
analytical solution of a differential problem, we must be content with
numerical solutions. As they are obtained for specified initial
values, it is important to know the stability behaviour of the
solutions for all initial values in the neighbourhood of a certain
equilibrium point.

For example, we consider a linear problem $\dot{\mathbf{y}} =
\mathbf{A} \mathbf{y}$, with exact solution $\mathbf{y}(t) =
\exp(\mathbf{A} t) \mathbf{y}_0$. This solution is analytically stable
if all trajectories remain bounded as $t \rightarrow \infty$. Theory
says that it is the case if and only if the real part of the
eigenvalues of $\mathbf{A}$ are strictly negative. If a numerical
solution of this problem is computed with the Euler method, the
system obtained is:
\begin{displaymath}
  \mathbf{y}(t^*+h) \approx \mathbf{y}(t^*) + \mathbf{A}h\mathbf{y}(t^*) =
  (\mathbf{I}+\mathbf{A}h)\mathbf{y}(t^*) =
  \mathbf{F} x(t^*)\enspace.
\end{displaymath}
In the same manner, the explicit Euler method is analytically stable if
the discretized system $\mathbf{y}_{k+1} = \mathbf{F} \mathbf{y}_k$ is
analytically stable.

Many classes of stability exist, such as A-stability, B-stability,
A($\alpha$)-stability, Algebraic stability; see~\cite{HNW93} for more
details. Regarding the linear example above, each stability class is associated with
a particular class of problems.

\subsection{Linear Stability}
\label{sec:linear-stability}
We focus on linear stability for explicit methods, which is easier
to study, and is enough to justify the use of interval coefficients. For
linear stability, the classical approach consists of computing the
\emph{stability domain} of the method (another well-known method being the root locus analysis). The \emph{stability function}
of explicit methods is given in \cite{HNW93}:
\begin{equation}
  R(z) = 1 + z \sum_j b_j + z^2 \sum_{j,k} b_j a_{jk} +
  z^3 \sum_{j,k,l} b_j a_{jk}a_{kl} + \dots\enspace,
\end{equation}
which can be written if the Runge--Kutta method is of order $p$ as
\begin{equation}
  R(z) = 1 + z + \frac{z^2}{2!} + \frac{z^3}{3!}  + \dots +
  \frac{z^p}{p!} + \mathcal{O}(z^{p+1})\enspace.
\end{equation}

For example, the stability function for a fourth-order method with
four stages, such as the classic RK4 method given in Figure~\ref{fig:butcher-tableau-rk4}, is:
\begin{equation}
  \label{eq:stability-function}
 R(z) = 1 + z + \frac{z^2}{2} + \frac{z^3}{6}  + \frac{z^4}{24}.
\end{equation}
The stability domain is then defined by $S = \{ z \in \mathcal{C}:
|R(z)| \leqslant 1\}$. This definition of $S$ can be transformed
into a constraint on real numbers following an algebraic process on
complex numbers, such as
\begin{displaymath}
  S = \{ (x,y): Re\left(\sqrt{Re(R(x+iy))^2 + Im(R(x+iy))^2}\right) \leqslant 1\}.
\end{displaymath}
The constraint produced is given in Equation~(\ref{ct_stab}).
\begin{multline}
  \label{ct_stab}
  \left( \frac{1}{6} x^3 y + \frac{1}{2} x^2 y - \frac{1}{6} x y^3 + x
  y - \frac{1}{6} y^3 +\right.
  \\
  \left.
  y^2 + \frac{1}{24} x^4 + \frac{1}{6} x^3 - \frac{1}{4} x^2 y^2 +
  \frac{1}{2} x^2 - \frac{1}{2} x y^2 + x + \frac{1}{24} y^4 -
  \frac{1}{2} y^2 + 1\right)^{\frac{1}{2}} \leqslant 1
\end{multline}

The set $S$ is now defined by a constraint on real numbers $x,y$ and
can be easily computed by a classical paving method \cite{JKDW01}.  The
result of this method is marked in blue in in Figure~\ref{fig:pav_int} for
an explicit Runge--Kutta fourth-order method with four stages, such as
the classical Runge--Kutta method (RK4).

We can study the influence of the numerical accuracy on the linear
stability. If we compute the coefficients (for example $1/6$ and
$1/24$) with low precision (even exaggeratedly in our case), the
stability domain is reduced as shown in
Figure~\ref{fig:pav_int}. 

First, we consider an error of \num{1e-8},
which is the classical precision of floating-point numbers for some
tools (see Figure~\ref{fig:pav_int} on the left). For example, the coefficient equal in theory to $1/6$ is encoded by $0.16666667$. 
Then, we consider
an error of \num{0.1} for this example, to see the impact: the
stability domain becomes the same as a first order method  such as
Euler's method. If it seems to be exaggerated, in fact it is not rare to find
old implementations of Runge--Kutta with only one decimal digit of accuracy (see
Figure~\ref{fig:pav_int} on the right).

\begin{figure}[ht!]
\centering
\begin{tabular}{c c}
 \includegraphics[width=5.5cm,keepaspectratio=true]{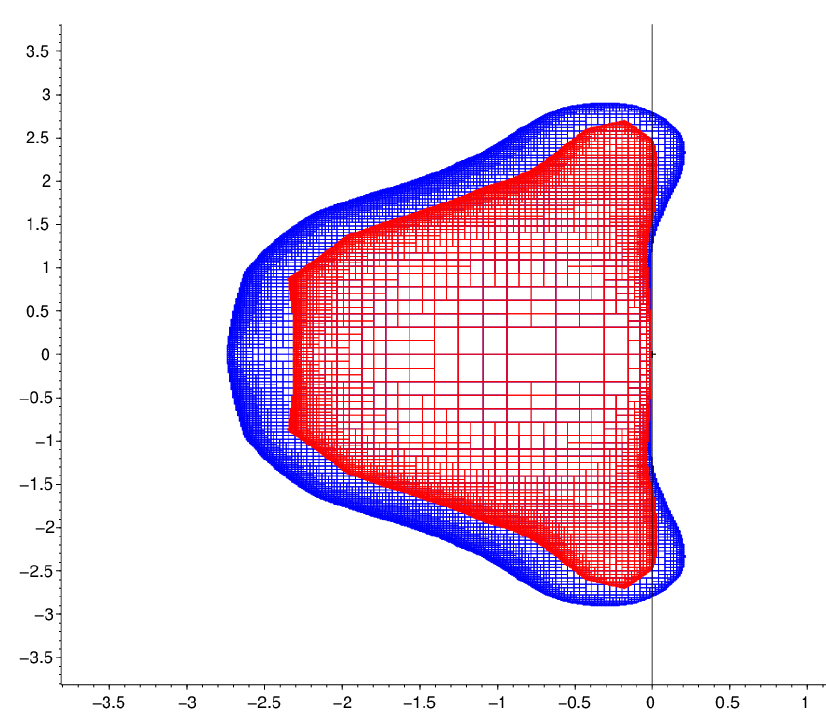} & \includegraphics[width=5.5cm,keepaspectratio=true]{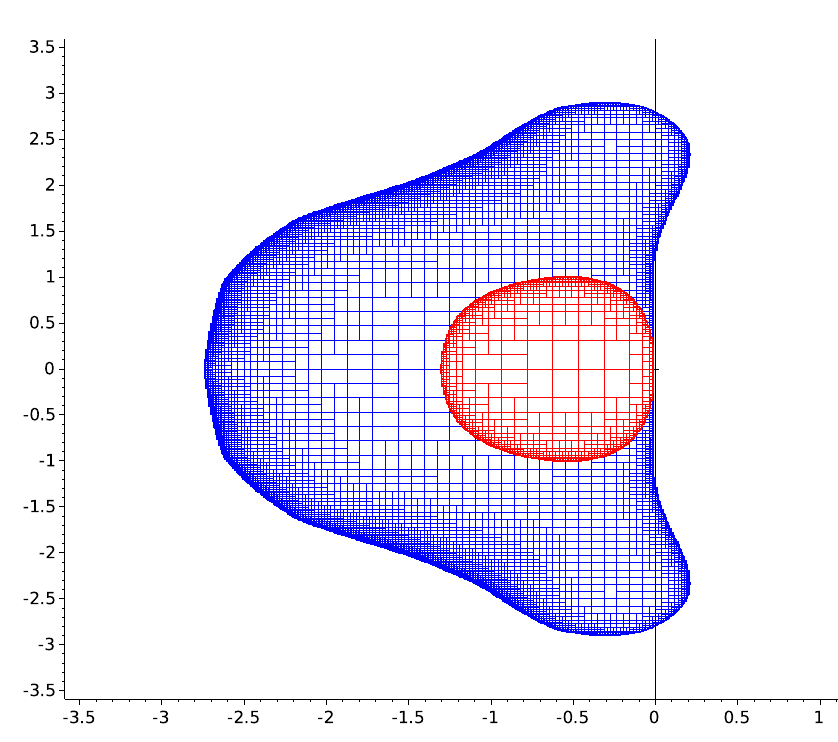}
\end{tabular}
\caption{Paving of stability domain for RK4 method with high precision
  coefficients (blue) and with small error (red) on coefficients
  (left) and large error on coefficients (right).}
 \label{fig:pav_int}
\end{figure}

\subsection{Algebraic Stability}
\label{sec:algebraic-stability}

Another interesting stability class for Runge--Kutta methods is algebraic
stability, which is useful for stiff problems or to solve
algebraic-differential equations. A method is algebraically stable if
the coefficients $a_{ij}$ and $b_i$ in the Butcher tableau are such
that
\begin{align*}
  b_i & \ge 0, \forall i=1, \dots, s: ~ 
  \mathbf{M} = (m_{ij}) = \left(b_ia_{ij} + b_ja_{ji} -
    b_ib_j\right)^s_{i,j=1} \text{is non-negative definite}.
\end{align*}

The test for non-negative definiteness can be done with constraint
programming by solving the eigenvalue problem $\det(\mathbf{M}-\lambda
\mathbf{I}) = 0$ and proving that $\lambda > 0$. $\mathbf{I}$ denotes
the identity matrix of dimension $s \times s$. For example, with a three
stage Runge--Kutta method, $s=3$, the constraint is:
\begin{multline}
  \label{eq:eigen}
  (m_{11}-\lambda)((m_{22}-\lambda)(m_{33}-\lambda) - m_{23}m_{32}) -
  m_{12}(m_{21}(m_{33}-\lambda) - m_{23}m_{13}) +
  \\
  m_{31}(m_{21}m_{32} - (m_{22}-\lambda)m_{31}) = 0.
\end{multline}

Based on a contractor programming approach
\cite{chabert2009contractor}, the CSP to solve is:
\begin{displaymath}
  \text{Equation~\eqref{eq:eigen} has no solution in }]-\infty,0[ \quad
  \equiv \quad \mathbf{M} \quad\text{is non-negative definite.}
\end{displaymath}
A contractor based on the Forward/Backward algorithm is applied to the
initial interval $[\num{-1e8},0]$; if the result obtained is the empty
interval, then Equation~\eqref{eq:eigen} has no negative solution, and
$\mathbf{M}$ is non-negative definite, so the 
method is algebraically stable.

We apply this method to the three-stage Lobatto IIIC, and the result
of contractor is empty, proving there is no negative eigenvalue, hence
the matrix $\mathbf{M}$ is non-negative definite and the Lobatto IIIC
method is algebraically stable, which is consistent with the theory. 
Similarly, we apply it to the three-stage Lobatto IIIA, and the
contractor finds at least one negative eigenvalue (\num{-0.0481125})
so this method is not algebraically stable, which is also
consistent with the theory.

Now, if an algebraically stable method is implemented with
coefficients in floating-point numbers, this property is lost. Indeed,
an error of $\num{1e-9}$ on $a_{ij}$ is enough to lose the
algebraic stability for Lobatto IIIC methods (a negative
eigenvalue appears equal to \num{-1.03041e-05}).

\subsection{Symplecticity}
\label{symplecticity}
Finally, another property of Runge--Kutta methods is tested, the
symplecticity. This property is associated with a notion of energy
conservation. A numerical solution obtained with a symplectic method
preserves an energy quantity, without formally expressing the
corresponding law. 

\begin{definition}[Symplectic integration methods]
Hamiltonian systems, given\\ 
by
 \begin{equation}
 \label{hamil}
  \dot{p}_i=-\frac{\partial H}{\partial q_i}(p,q), \quad \dot{q}_i=\frac{\partial H}{\partial p_i}(p,q),
 \end{equation}
have two remarkable properties: i) the solutions preserve the Hamiltonian $H(p,q)$; ii) the corresponding flow is symplectic, strictly speaking, preserves the differential 2-form $\omega^2=\sum_{i=1}^n dp_i \wedge dq_i$. %In general, a numerical method used to solve Equation~\eqref{hamil} destroys these properties. 
A numerical method used to solve Equation~\eqref{hamil}, while preserving these properties, is a {\bf symplectic integration method}. 
\end{definition}

\begin{definition}[Symplectic interval methods]
 A Runge--Kutta method with interval coefficients $\{[\mathbf{b}],[\mathbf{c}],[\mathbf{A}]\}$, such that a method defined by $\{\mathbf{b},\mathbf{c},\mathbf{A}\}$ with $\mathbf{b} \in [\mathbf{b}]$, $\mathbf{c} \in [\mathbf{c}]$, and $\mathbf{A}  \in [\mathbf{A}]$ is symplectic, is a {\bf symplectic interval method}.
\end{definition}

\noindent A Runge--Kutta method is symplectic if it satisfies the condition 
$\mathbf{M}=0$, where
\begin{displaymath}
  \mathbf{M} = (m_{ij}) = (b_ia_{ij} + b_ja_{ji} -
  b_ib_j)^s_{i,j=1}.
\end{displaymath}

With interval computation of $\mathbf{M}$, it is possible to verify if
$0 \in \mathbf{M}$, which is enough to prove that the method with
interval coefficients is symplectic. Indeed, it is sufficient to prove
that  a trajectory which
preserves a certain energy conservation condition exists inside the numerical solution.

We apply this approach to the three-stage Gauss-Legendre method with coefficients computed with interval arithmetic. The
matrix $\mathbf{M}$ contains the zero matrix (see
Equation~\eqref{symp0}), so this method is symplectic, which is
in agreement with the theory.

\begin{equation}
  \label{symp0}
  \mathbf{M} =
  \begin{pmatrix}
    [-1.3e^{-17}, 1.4e^{-17}] & [-2.7e^{-17}, 2.8e^{-17}] & [-2.7e^{-17}, 1.4e^{-17}]\\
    [-2.7e^{-17}, 2.8e^{-17}] & [-2.7e^{-17}, 2.8e^{-17}] & [-1.3e^{-17}, 4.2e^{-17}]\\
    [-2.7e^{-17}, 1.4e^{-17}] & [-1.3e^{-17}, 4.2e^{-17}] & [-1.3e^{-17}, 1.4e^{-17}]
  \end{pmatrix}
\end{equation}

Now, if we compute only one term of the Gauss-Legendre method with
floating-point numbers, for example $a_{1,2} = 2.0/9.0 -
\sqrt{15.0}/15.0$, the symplecticity property is lost (see
Equation~\eqref{sympnot0}).
\begin{multline}
  \mathbf{M} = \\
  \begin{pmatrix}
    [-1.3e^{-17}, 1.4e^{-17}] & {\bf [-1.91e^{-09}, -1.92e^{-09}]} & [-2.7e^{-17}, 1.4e^{-17}]\\
    {\bf [-1.91e^{-09}, -1.92e^{-09}]} & [-2.7e^{-17}, 2.8e^{-17}] & [-1.3e^{-17}, 4.2e^{-17}]\\
    [-2.7e^{-17}, 1.4e^{-17}] & [-1.3e^{-17}, 4.2e^{-17}] & [-1.3e^{-17}, 1.4e^{-17}]
  \end{pmatrix}\label{sympnot0}
\end{multline}

\begin{note}
 Butcher suggested to add the constraints $b_ia_{ij} + b_ja_{ji} -  b_ib_j = 0, \forall {i,j=1..s}$ to the order conditions to directly compute a symplectic method, by construction. 
\end{note}

\section{A Constraint Optimization Approach to Define New Runge--Kutta
  Methods}
\label{sec:cp_cost}

In the previous section, the properties of Runge--Kutta methods with
interval coefficients in the Butcher tableau have been studied, and we
have shown that these properties are preserved with intervals while they are
often lost with floating-point numbers. In this section, an approach
based on constraint optimization is presented to obtain optimal Runge--Kutta
methods with interval coefficients. The cost function is also discussed, while 
the solving procedure is presented in Section~\ref{subsec:details}.

\subsection{Constraints}

The constraints to solve to obtain a novel Runge--Kutta method are the
ones presented in Section~\ref{sec:butcher-theory-Runge--Kutta}, and the
approach is based on a CSP solver based on contractors and a branching
algorithm (see Section~\ref{sec:rk_interval}). The problem under consideration
can be under-constrained, and more than one solution can exist (for
example, there are countless fully implicit fourth-order methods with
three stages). With the interval analysis approach, which is based on set
representation, a continuum of coefficients can be obtained. As the
coefficients of the Butcher tableau have to be as tight as possible to
obtain sharp enclosure of the numerical solution, a continuum (or more
than one) of solutions is not serviceable. Indeed, in a set of
solutions, or a continuum, it is interesting to find an optimal
solution with respect to a given cost.

Note that using the framework of CPS, adding a cost function and hence
solving a constraint optimization problem can be done following
classical techniques such as those defined in
\cite{hansen_global_2003}.

\subsection{Cost function}
\label{cost}

In the literature, a cost function based on the norm of the local truncation error is sometimes chosen \cite{ralston}.

\subsubsection{Minimizing the LTE}

There exist many explicit second-order methods with two stages. A
general form, shown in Table~\ref{erk2}, has been defined. With
$\alpha=1$, this method is Heun's method, while $\alpha=1/2$ gives the
midpoint method (see \cite{Butcher63} for details about these methods).

\begin{table}[htb]
\centering
\begin{tabular}{c | c c}
 0 & 0 &\\
 $\alpha$ & $\alpha$ &\\
 \hline
 & 1-1/(2$\alpha$) & 1/(2$\alpha$)\\
\end{tabular}
\caption{General form of ERK with 2 stages and order 2}
\label{erk2}
\end{table}

Ralston has proven that $\alpha=2/3$ minimizes the sum of square of
coefficients of rooted trees in the local truncation error computation
\cite{butcher2003numerical}, which is given by:

\begin{equation}
 \min_{\alpha} (-3 \alpha/2 + 1)^2 + 1.
\end{equation}

The resulting Butcher tableau is given in Table~\ref{ralston}.
\begin{table}[htb]
\centering
\begin{tabular}{c | c c}
 0 & 0 &\\
 2/3 & 2/3 &\\
 \hline
 & 1/4 & 3/4\\
\end{tabular}
\caption{Ralston method}
\label{ralston}
\end{table}

\subsubsection{Maximizing order}
\label{sec:ralston}

Another way to obtain a similar result is to try to attain one order
larger than the desired one. For example, if, as Ralston, we try to build an
explicit second-order method with two stages but as close as possible
to the third order by minimizing:
\begin{equation}
 \min_{a_{ij},b_i,c_i} \left(\sum b_i c_j a_{ij} - \frac{1}{6}\right)^2 +
 \left(\sum b_i c_i^2 - \frac{1}{3}\right)^2\enspace.
\end{equation}
The same result is obtained (Table~\ref{ralston_int}). This way of
optimization is more interesting for us because it reuses the
constraint generated by the order condition. It also minimizes the LTE
at a given order $p$, because it tends to a method of order $p+1$ which
has a LTE equal to zero at this order. It is important to note that
minimizing the LTE or maximizing the order leads to the same result;
the difference is in the construction of the cost function and in the
spirit of the approach.

\begin{table}[htb]
  \centering
  \caption{Ralston method with interval coefficients}
  \label{ralston_int}
  \begin{tabular}{c | c c}
    $[-0,0]$ & $[-0,0]$ &
    \\
    $0.6...6[6,7]$ & $0.6...6[6,7]$ &
    \\
    \hline
    & $[0.25,0.25]$ & $[0.75,0.75]$
    \\
  \end{tabular}
\end{table}

\section{Experiments}
\label{sec:expe}

Experiments are performed to, first, re-discover 
Butcher's theory and, second, to find new methods with desired structure.

\subsection{Details of Implementation}
\label{subsec:details}

To implement the approach presented in this paper, two steps need to be 
performed. The first one is a formal procedure used to generate the CSP, 
and the second one is applying a CSP solver based on interval analysis. 

\subsubsection{Definition of the Desired Method and Generation of the CSP}
The definition of the desired method consists of the choice of
\begin{itemize}
\item Number of stages of the method
\item Order of the method
\item Structure of the method (singly diagonal, explicit method, diagonally
implicit
  method, explicit first line and/or stiffly accurate method)
\end{itemize}

Based on this definition and the algorithm defined in \cite{Bornemann01}, a formal 
procedure generates the constraints associated with the structure and Butcher rules
(see Section~\ref{sec:butcher-theory-Runge--Kutta}), and eventually a cost function (see Section~\ref{sec:ralston}). 

\subsubsection{Constraint Programming and Global Optimization}

Problem solution is done with Ibex, a library for interval
computation with a constraint solver and a global optimizer.

This library can address two major problems \cite{ibex_website}:
\begin{itemize}
\item System solving: A guaranteed enclosure for each solution of a
  system of (non-linear) equations is calculated.
\item Global optimization: A global minimizer of some function under
  non-linear constraints is calculated with guaranteed bounds on the
  objective minimum.
\end{itemize}

Global optimization is performed with an epsilon relaxation, so the
solution is optimal but the constraints are satisfied with respect to~the
relaxation. A second pass with the constraint solver is then needed to
find the validated solution inside the inflated optimal solution. The
solver provides its result in the form of an interval vector such
as $([b_i],[c_i],[a_{ij}])$.

Some experiments are performed in the following. First, the
constraint solving part, which allows us to find methods with
sufficient constraints to be the unique solution, is
tested. Second, the global optimizer is used to find the
optimal methods which are under-constrained by order conditions. Both parts are used to find the existing methods and 
potentially new ones. In the following, just few methods that can be
computed are shown. Indeed, numerous methods can be obtained. 

\subsection{Constraint Solving}
The first part of the presented approach is applied. It allows one to
solve the constraints defined during the user interface process,
without cost function. This option permits 
\begin{itemize}
 \item finding a method if there
is only one solution (well-constrained problem),
\item knowing if there is
no solution available,
\item validating the fact that there is a
continuum in which an optimum can be found.
\end{itemize}
  To demonstrate the
efficiency of this solution part, we apply it with user choices that
lead to existing methods and well-known results. After that, we describe some new interesting methods.

\subsubsection{Existing Methods}

\paragraph{Only One Fourth-Order Method with Two Stages:
  Gauss-Legendre}

If we are looking for a fourth-order fully implicit method with two
stages, the theory says that only one method exists, the
Gauss-Legendre scheme.  In the following, we try to obtain the same
result with the solution part of our scheme.

The CSP for this method is defined as follows:
\begin{align*}
  \mathcal{X} & = \{\mathbf{b}, \mathbf{c}, \mathbf{A}\}
  \\
  \mathcal{D} & = \{[-1,1]^2,[0,1]^2,[-1,1]^4\}
  \\
  \mathcal{C} & = \left(\begin{aligned} b_0 + b_1 - 1 & = 0
    \\
    b_0c_0 +b_1c_1 - \frac{1}{2} & = 0
    \\
    b_0(c_0)^2 +b_1(c_1)^2 - \frac{1}{3} & = 0
    \\
    b_0a_{00}c_0 +b_0a_{01}c_1 + b_1a_{10}c_0 +b_1a_{11}c_1 -
    \frac{1}{6} & = 0
    \\
    b_0(c_0)^3 +b_1(c_1)^3 - \frac{1}{4} & = 0
    \\
    b_0c_0a_{00}c_0 +b_0c_0a_{01}c_1 + b_1c_1a_{10}c_0
    +b_1c_1a_{11}c_1 - \frac{1}{8} & = 0
    \\
    b_0a_{00}(c_0)^2 +b_0a_{01}(c_1)^2 + b_1a_{10}(c_0)^2
    +b_1a_{11}(c_1)^2 - \frac{1}{12} & = 0
    \\
    b_0a_{00}a_{00}c_0 +b_0a_{00}a_{01}c_1 + b_0a_{01}a_{10}c_0
    +b_0a_{01}a_{11}c_1 + b_1a_{10}a_{00}c_0 + &
    \\
    b_1a_{10}a_{01}c_1 + b_1a_{11}a_{10}c_0 +b_1a_{11}a_{11}c_1 -
    \frac{1}{24} & = 0
    \\
    a_{00}+a_{01}-c_0 & = 0
    \\
    a_{10}+a_{11}-c_1 & = 0
    \\
    c_0 & < c_1
  \end{aligned}
\right)
\end{align*}

The result from the solver is that there is only one solution, and if
this result is written in the Butcher tableau form
(Table~\ref{gauss2}), we see that this method is a numerically
guaranteed version of Gauss-Legendre.

\begin{table}[ht]
  % \small
  \centering
  \caption{Guaranteed version of Gauss-Legendre}
  \label{gauss2}
  \begin{tabular}{c | c c}
    $0.21132486540[5,6]$ & $[0.25, 0.25]$ &$-0.038675134594[9,8]$\\
    $0.78867513459[5,6]$ & $0.53867513459[5,6]$ &$[0.25, 0.25]$\\
    \hline
    & $[0.5, 0.5]$ & $[0.5, 0.5]$
  \end{tabular}
\end{table}

\paragraph{No Fifth-Order Method with Two Stages}

It is also easy to verify that there is no fifth-order methods with
two stages. The CSP generated is too large to be presented here. The
solver proves that there is no solution, in less than 0.04 seconds.

\paragraph{Third-Order SDIRK Method with Two Stages}

The solver is used to obtain a third-order Singly Diagonal Implicit Runge--Kutta (SDIRK) method with two
stages. The result obtained is presented in Table~\ref{sdirk}. This
method is known; it is the SDIRK method with $\lambda=\frac{1}{2} - \frac{1}{6}\sqrt{3}$.
\begin{table}[htb]
  \centering
  \caption{Third-order SDIRK method with two stages ($c_1 < c_2$)}
  \label{sdirk}
  \begin{tabular}{c | c c}
    $0.21132486540[5,6]$ & $0.21132486540[5,6]$ & $[0, 0]$\\
    $0.78867513459[5,6]$ & $0.577350269[19,20]$ & $0.21132486540[5,6]$\\
    \hline
    &$[0.5, 0.5]$ & $[0.5, 0.5]$
  \end{tabular}
\end{table}

\begin{note}
Butcher suggested to remove the constraints on the $c_i$, which impose the growth for these latter, to obtain the other SDIRK method with two stages. It was done and two solutions are found: the one corresponding to Table~\ref{sdirk} and the method given in Table~\ref{sdirk2}, which corresponds to $\lambda=\frac{1}{2} + \frac{1}{6}\sqrt{3}$.

\begin{table}[htb]
  \centering
  \caption{Third-order SDIRK method with two stages ($c_1 > c_2$)}
  \label{sdirk2}
  \begin{tabular}{c | c c}
    $0.78867513459[5,6]$ & $0.78867513459[5,6]$ &  $[0, 0]$\\
    $0.21132486540[5,6]$ & $\minus 0.577350269[19,20]$ & $0.78867513459[5,6]$\\   
    \hline
    &$[0.5, 0.5]$ & $[0.5, 0.5]$
  \end{tabular}
\end{table}

 %compute without order of c_i to find the other one
% first sol=([0.4999999999999987, 0.5000000000000014] ; [0.4999999999999986, 0.5000000000000013] ; [0.7886751345948123, 0.7886751345948134] ; [0.2113248654051857, 0.2113248654051884] ; [0.7886751345948123, 0.7886751345948134] ; [0, 0] ; [-0.5773502691896266, -0.577350269189625] ; [0.7886751345948123, 0.7886751345948134])
% 
% second sol=([0.4999999998592928, 0.5000000001131641] ; [0.4999999998875378, 0.5000000001393785] ; [0.2113248653682223, 0.2113248654275329] ; [0.7886751344724929, 0.7886751347013814] ; [0.2113248653682223, 0.2113248654275329] ; [0, 0] ; [0.57735026904496, 0.5773502693331591] ; [0.2113248653682223, 0.2113248654275329])

\end{note}

\subsubsection{Other Methods}
Now, it is possible to obtain new methods with the presented approach.

\begin{remark}
  It is hard to be sure that a method is new because there is no
  database collecting all the methods.
\end{remark}

\paragraph{A Fourth-Order Method with Three Stages, Singly and Stiffly
  Accurate}

In theory, this method is promising because it has the capabilities,
desirable for stiff problems (and for differential algebraic equations),
to simultaneously optimize the Newton's method solution process and to be stiffly accurate (to be more
efficient with respect to stiffness). Our approach finds a unique method, unknown to-date,
satisfying to these requirements.  The
result is presented in Table~\ref{s3o4}.

\begin{table}[htb]
  \footnotesize
  \centering
  \caption{A fourth-order method with three stages, singly and stiffly accurate: S3O4}
  \label{s3o4}
  \begin{tabular}{@{\extracolsep{-1pt}}c | l l r}
    $0.1610979566[59,62]$ &$0.105662432[67,71]$ & $0.172855006[54,67]$ & $-0.117419482[69,58]$\\
    $0.655889341[44,50]\phantom{1}$ &$0.482099622[04,10]$ & $0.105662432[67,71]$ & $0.068127286[68,74]$\\
    $[1, 1]$& $0.3885453883[37,75]$ & $0.5057921789[56,65]$ & $0.105662432[67,71]$\\
    \hline
    &$0.3885453883[37,75]$ & $0.5057921789[56,65]$ & $0.105662432[67,71]$
  \end{tabular}
\end{table}

\begin{note}
 The Singly property has no interest in this case. It is just to show that the constraint of equality for the diagonal coefficients can be taken into account. 
\end{note}

\paragraph{A Fifth-Order Method with Three Stages, Explicit First
  Line}
With only $6$ non zero coefficients in the intermediate computations,
this method could be a good compromise between a fourth-order method
with four intermediate computations (fourth-order Gauss-Legendre) and
sixth-order with nine intermediate computations (sixth-order
Gauss-Legendre). As we know, there is no Runge--Kutta method with the
same capabilities as the Gauss-Legendre method, but with fifth order.
The result is presented in Table~\ref{s3o5}.

\begin{table}[htb]
  \footnotesize
  \centering
  \caption{A fifth-order method with three stages, explicit first line: S3O5}
  \label{s3o5}
  \begin{tabular}{c | c c c}
    $[0, 0]$ &$[0, 0]$ & $[0, 0]$ & $[0, 0]$ \\
    $0.355051025[64,86]$ &$0.152659863[17,33]$ & $0.220412414[50,61]$ & $-0.0180212520[53,23]$ \\
    $0.844948974[23,34]$&$0.087340136[65,87]$ & $0.57802125[20,21]$ & $0.179587585[44,52]$ \\
    \hline
    & $0.111111111[03,26]$ & $0.512485826[00,36]$ & $0.376403062[61,80]$
  \end{tabular}
\end{table}

\begin{note}
 Butcher noticed that this method is not new. Indeed, it is the Radau I method at order 5, this method can be found in \cite{butcher2003numerical}. 
\end{note}

\subsection{Global Optimization}

When  the first part of our solution process provides more than one
solution or a continuum of solutions, we are able to define an
optimization cost to find the best solution with respect to that
cost. We have decided to use a cost which implies that the method tends to
a higher order (Section~\ref{cost}).

\subsubsection{Existing Methods}

\paragraph{Ralston}

We obtain the same result as the one published by Ralston in \cite{ralston}, and described in
Section~\ref{sec:ralston}.

\paragraph{Infinitely many Second-Order Methods with Two Stages, Stiffly
  Accurate and Fully Implicit}
The theory says that there are infinitely many second-order methods with two
stages, stiffly accurate and fully implicit. But there is only one
third-order method: radauIIA.

The generated CSP for this method is defined as follows:
\begin{align*}
  \mathcal{X} & = \left\{\mathbf{b}, \mathbf{c}, \mathbf{A}\right\}
  \\
  \mathcal{D} & = \left\{[-1,1]^2,[0,1]^2,[-1,1]^4\right\}
  \\
  \mathcal{C} & = \left(\begin{aligned}
    b_{0} +b_{1} - 1 & \leqslant\varepsilon
    \\
    b_{0} +b_{1} -1 & \geqslant -\varepsilon
    \\
    b_{0}c_{0} +b_{1}c_{1} - \frac{1}{2} & \leqslant\varepsilon
    \\
    b_{0}c_{0} +b_{1}c_{1} - \frac{1}{2} & \geqslant -\varepsilon
    \\
    a_{00}+a_{01}-c_{0} & = 0
    \\
    a_{10}+a_{11}-c_{1} & = 0
    \\
    c_{0} & \leqslant c_{1}
    \\
    a_{10} - b_{0} & = 0
    \\
    a_{11} - b_{1} & = 0
  \end{aligned}
\right)
  \\
  \text{Minimize} &\; \left(b_{0} (c_{0})^2 +b_{1} (c_{1})^2 -
    \frac{1}{3}\right)^2 + \left(b_{0} a_{00} c_{0} +b_{0} a_{01}
    c_{1} + b_{1} a_{10} c_{0} +b_{1} a_{11} c_{1} -
    \frac{1}{6}\right)^2
\end{align*}

The optimizer find an optimal result in less than 4 seconds; see
Figure~\ref{radau_op}.
\begin{table}[htb]
  % \footnotesize
  \centering
  \caption{Method close to RadauIIA obtained by optimization}
  \label{radau_op}
  \begin{tabular}{c | c c}
    $0.333333280449$ & $0.416655823215$ & $-0.0833225527662$
    \\
    $0.999999998633$ & $0.749999932909$ & $0.250000055725$
    \\
    \hline
    & $0.749999939992$ & $0.250000060009$
  \end{tabular}
\end{table}

The cost of this solution is in $[-\infty,\num{2.89e-11}]$, which
means that $0$ is a possible cost, that is to say that a third-order
method exists.  A second pass with the solver is needed to find the
acceptable solution (without relaxation) by fixing some coefficients
($b_1 = 0.75$ and $c_2=1$ for example); the  well known RadauIIA
method is then obtained.

\subsubsection{Other Methods}
Now, we are able to obtain new methods with our optimizing procedure.

\paragraph{An Optimal Explicit Third-Order Method with Three Stages}
There are infinitely many explicit $(3,3)$-methods, but there is no
fourth-order method with three stages. Our optimization process helps
us to produce a method as close as possible to fourth order (see
Table~\ref{erk33o}). The corresponding cost is computed to be in
$0.00204[35,49]$. As explained before, this method is not validated
due to relaxed optimization. We fix some coefficients (enough to
obtain only one solution) by adding the constraints given in
Equation~\ref{add_cst}. After this first step, the solver is used to
obtain a guaranteed method, close to the fourth order (see
Table~\ref{erk33}).

\begin{equation}
  \label{add_cst}
  \left\{\begin{aligned}
      b_1 & > 0.195905;\\
      b_1 & < 0.195906;\\
      b_2 & > 0.429613;\\
      b_2 & < 0.429614;\\
      b_3 & > 0.37448000;\\
      b_3 & < 0.37448001;\\
      c_2 & > 0.4659;\\
      c_2 & < 0.4660;\\
      c_3 & > 0.8006;\\
      c_3 & < 0.8007;\\
      a_{32} & > 0.9552;\\
      a_{32} & < 0.9553;\\
      a_{31} & > -0.1546;\\
      a_{31} & < -0.1545;
    \end{aligned}\right.
\end{equation}

\begin{table}[htb]
  \footnotesize
  \centering
  \caption{An optimal explicit third-order method with three stages
    (not validated due to relaxation)}
 \label{erk33o}
 \begin{tabular}{c | c c c}
   1.81174261766e-08& 6.64130952624e-09 & 9.93482546211e-09 & -1.11126730095e-09 \\
   0.465904769163&0.465904768843 & -1.07174862901e-09 & 3.94710325991e-09 \\
   0.800685593936& -0.154577204301 & 0.955262788613 & 9.99497058355e-09 \\
    \hline
    &0.195905959102 & 0.429613967179 & 0.37448007372
  \end{tabular}
 \end{table}

\begin{table}[htb]
  \footnotesize
  \centering
  \caption{A guaranteed explicit third-order method with three stages, the
    closest to fourth-order}
  \label{erk33}
  \begin{tabular}{c | c c c}
    [0, 0] &[0, 0] & [0, 0] & [0, 0]\\
    $0.4659048[706,929]$ &$0.4659048[706,929]$ & [0, 0] & [0, 0]\\
    $0.8006855[74,83]$&$-0.154577[20,17]$ & $0.9552627[48,86]$ & [0, 0]\\
    \hline
    &$0.19590[599,600]$ & $0.42961[399,400]$ & $0.3744800[0,1]$
  \end{tabular}
\end{table}

If we compute the order conditions up to fourth order, we verify that
this method is third-order by inclusion, and close to fourth-order.  We
compute the Euclidean distance between order condition and obtained
values. For our optimal method the distance is
$0.045221[2775525,3032049]$ and for Kutta(3,3) \cite{kutta1901beitrag}, which is known to be
one of the best explicit (3,3) method\footnote{``Von den neueren Verfahren halte ich das folgende von Herrn Kutta angegebene f\"ur das beste.``, C.Runge 1905 \cite{HNW93}}, $0.058926$. Our method is then
closer to fourth order than Kutta(3,3).  As far as we know, this method is
new.

\begin{table}[htb]
  \centering
  \caption{Order conditions up to fourth order}
  \begin{tabular}{c | c c}
    Order & Result of optimal method & Order condition\\
    \hline
    Order 1 & [0.99999998, 1.00000001] & 1\\
    \hline
    Order 2 & [0.499999973214, 0.500000020454] & 0.5\\
    \hline
    Order 3 & [0.33333330214, 0.333333359677] & 0.333333333333\\
    Order 3 & [0.166666655637, 0.166666674639] & 0.166666666667\\
    \hline
    Order 4 & [0.235675128044, 0.235675188505] & 0.25\\
    Order 4 & [0.133447581964, 0.133447608305] & 0.125\\
    Order 4 & [0.0776508066238, 0.0776508191916] & 0.0833333333333\\
    Order 4 & [0, 0] & 0.0416666666667\\
  \end{tabular}
  \label{ord_4}
\end{table}

\begin{figure}[ht!]
  \centering
  \includegraphics[width=9cm,keepaspectratio=true]{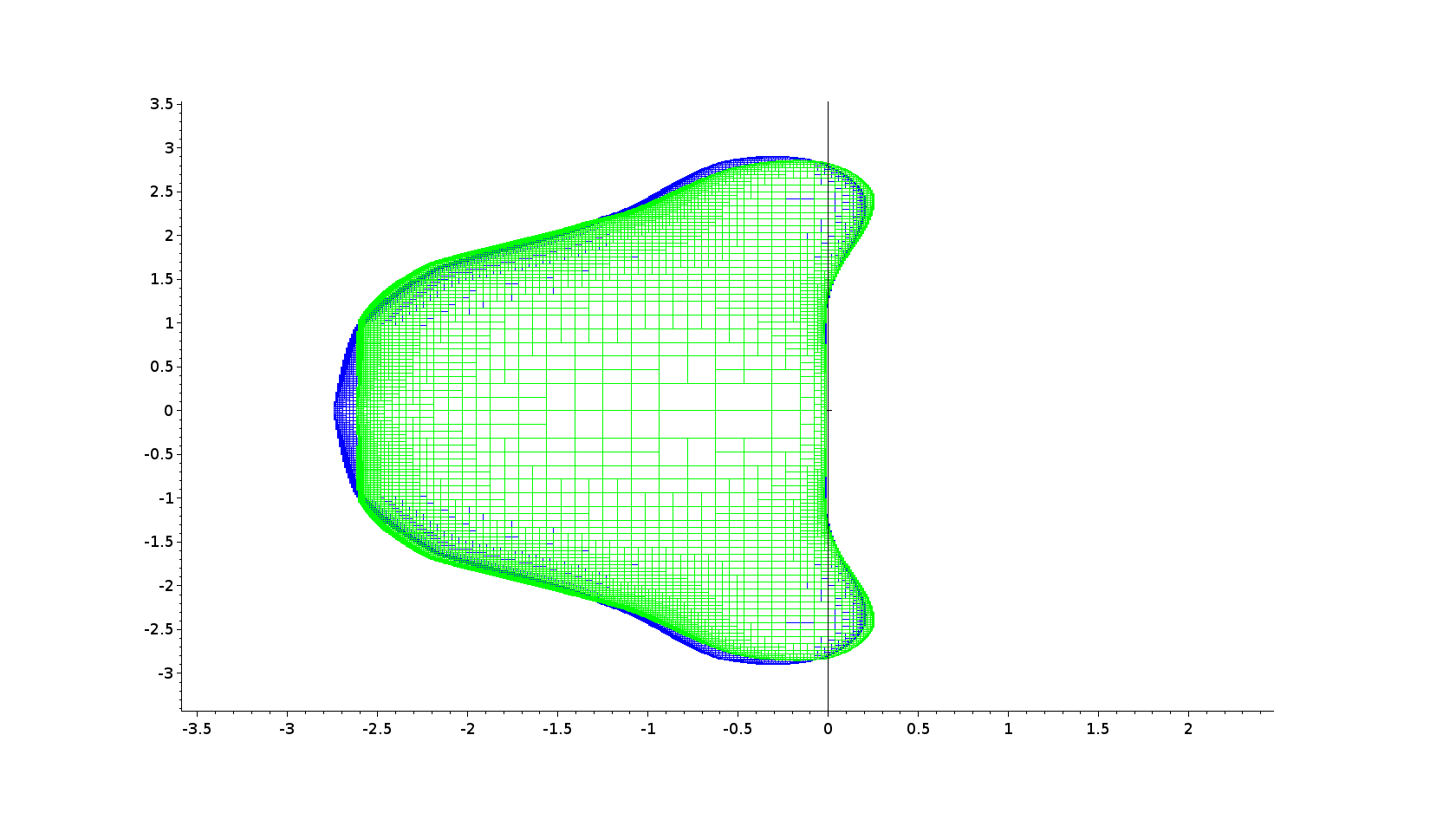}
  \caption{Paving of stability domain for RK4 method with high
    precision coefficients (blue) and for ERK33 (green).}
  \label{fig:pav_erk33}
\end{figure}

\section{Implementation  in the DynIBEX Library}
\label{sec:application}

DynIBEX offers a set of validated numerical integration methods based
on Runge--Kutta schemes to solve initial value problem of ordinary
differential equations and for DAE in Hessenberg index $1$ form. Even
if our approach is  applied not only to validated integration but also to
classical numerical integration with interval coefficients, the
validated integration allows us to obtain a validated enclosure of the
final solution of the simulation. This enclosure provides, with its
diameter, a guaranteed measure of the performance of the integration
scheme. The computation time increases rapidly with respect to the order of the
method; because of the LTE, its complexity is $\mathcal{O}(n^{p+1})$,
with $n$ the dimension of the problem and $p$ the order. The experimental results
provide the sharpest enclosure of the final solution with the lowest
possible order.

We implement three new methods: S3O4 (Table~\ref{s3o4}), S3O5
(Table~\ref{s3o5}), and ERK33 (Table~\ref{erk33}).

\subsection{Experiments with S3O4}
The test is based on an oil reservoir problem,  a stiff
problem given by the initial value problem:
\begin{equation}
  \dot{\mathbf{y}} = \begin{bmatrix}
    \dot{y_0}\\
    \dot{y_1}
  \end{bmatrix}=
  \begin{bmatrix}
    y_1 \\
    y_1^2 - \frac{3}{\epsilon + y_0^2}
  \end{bmatrix}, \text{ with } y(0)=(10,0)^T \text{ and } \epsilon=\num{1e-4}
  \enspace.
\end{equation}
A simulation up to $t=40s$ is performed. This problem being stiff, the
results of the new method S3O4 are compared with the Radau family,
specially the RadauIIA of third and fifth order. The results are
summarized in Table~\ref{res_s3o4}.

\begin{table}[htb]
  \centering
  \caption{Results for S3O4}
  \label{res_s3o4}
  \begin{tabular}{c | c c c}
    Methods & time & no.~steps & norm of diameter of final solution \\
    \hline
    S3O4 & $39$ & $1821$  & \num{5.9e-5}
    \\
    Radau3 & $52$ & $7509$ & \num{2.e-4}
    \\
    Radau5 & $81$ & $954$ & \num{7.6e-5}
    \\
  \end{tabular}
\end{table}

S3O4 is a singly implicit scheme, to optimize the Newton's method solving, and
stiffly accurate, to be more efficient with respect to stiff problems. Based on
experimental results, S3O4 seems to be as efficient as the fifth-order method RadauIIA, 
but faster than the third-order method RadauIIA.

\subsection{Experiments with S3O5}
The test is based on an interval problem, which can quickly explode,
given by the initial value problem:
\begin{equation}
  \dot{\mathbf{y}} = \begin{bmatrix}
    \dot{y_0}\\
    \dot{y_1}\\
    \dot{y_2}
  \end{bmatrix}=
  \begin{bmatrix}
    1\\
    y_2\\
    \frac{y_1^3}{6}-y_1+2 \sin(\lambda y_0)
  \end{bmatrix}, \text{ with } y(0)=(0,0,0)^T \text{ and } \lambda \in
  [2.78,2.79]\enspace.
\end{equation}
A simulation up to $t=10s$ is performed. Since this problem includes an
interval parameter, a comparison with Gauss-Legendre family makes
sense, Gauss--Legendre methods have a good contracting property. Thus, we compare
to the fourth- and sixth-order Gauss-Legendre methods. Results
are summarized in Table~\ref{res_s3o5}.

\begin{table}[htb]
  \centering
  \caption{Results for S3O5}
  \label{res_s3o5}
  \begin{tabular}{c | c c c}
    Methods & time & no.~steps & norm of diameter of final solution \\
    \hline
    S3O5 & $92$ & $195$ & $5.9$
    \\
    Gauss4 & $45$ & $544$ & $93.9$
    \\
    Gauss6 & $570$ & $157$ & $7.0$
    \\
  \end{tabular}
\end{table}

The results show that S305 is more efficient than the sixth-order
Gauss-Legendre method and five time faster. Although the fourth-order
Gauss-Legendre method is two times faster, the final solution is much
wider.

\subsection{Experiments with ERK33}
The test is based on the classical Van der Pol problem, which contains
a limit circle, and is given by the initial value problem:
\begin{equation}
  \dot{\mathbf{y}} = \begin{bmatrix}
    \dot{y_0}\\
    \dot{y_1}
  \end{bmatrix}=
  \begin{bmatrix}
    y_1\\
    \mu (1-y_0^2) y_1 - y_0
  \end{bmatrix}, \text{ with } y(0)=(2,0)^T \text{ and } \mu=1\enspace.
\end{equation}
A simulation up to $t=10s$ is performed. Since this problem contains a
limit circle, it can be effectively simulated with an explicit
scheme. The two most famous schemes are the explicit Runge--Kutta  (RK4), the most
used, and Kutta, known to be the optimal explicit third-order
scheme. We compare ERK33 with these methods, and present the results in Table~\ref{res_erk33}.

\begin{table}[htb]
  \centering
  \caption{Results for ERK33}
  \label{res_erk33}
  \begin{tabular}{c | c c c}
    Methods & time & no.~steps & norm of diameter of final solution \\
    \hline
    ERK33 & $3.7$ & $647$ & \num{2.2e-5}
    \\
    Kutta(3,3) & $3.5$ & $663$ & \num{3.4e-5}
    \\
    RK4 & $4.3$ & $280$ & \num{1.9e-5}
    \\
  \end{tabular}
 \end{table}
 These results show that ERK33 is equivalent in time consumed but with
 performance closer to RK4.

\subsection{Discussion}

After experimentation with the three new Runge--Kutta methods obtained
with the constraint programming approach presented in this paper, it
is clear that these methods are effective.  Moreover, even with
coefficients of the Butcher tableau expressed in intervals with a
diameter of \num{1e-10} (for S3O4 described in Table~\ref{s3o4} and S3O5 described in Table~\ref{s3o5}) to
\num{1e-8} (for ERK33 described in Table~\ref{erk33}), the final solution is often
narrower for the same or higher order methods with exact
coefficients. A strong analysis is needed, but it seems that by
guaranteeing the properties of the method, the contractivity of the
integration schemes is improved.

\begin{note}
 As a global remark, Butcher suggested to combine the presented approach with algebraic knowledges. The positiveness of the coefficients $b_i$ can sometimes be used. Moreover, the $C(2)$ condition \cite{HNW93} could also provides an additional constraint. It is a promising improvement clue. 
\end{note}

\section{Conclusion}
\label{sec:conclusion}

In this paper, a new approach to discovering new
Runge--Kutta methods with interval coefficients has been presented. In a first step, we
show how interval coefficients can preserve
properties such as stability or symplecticity, unlike
coefficients expressed in floating-point numbers. We have presented two tools,
a CSP solver used to find the unique solution of the Butcher
rules, and an optimizer procedure to obtain the best method with respect to a
well chosen cost. This cost will provide a method of order $p$ with a LTE as close as possible to 
the LTE of a method at order $p+1$. 
Finally, the methods obtained guarantee that the desired order and
properties are obtained. These new methods are then implemented in a
validated tool called DynIbex, and some tests on problems well chosen
with respect to the required properties are performed. The results lead us to
conclude that the approach is valid and efficient in the sense that
the new methods provide highly competitive results with respect to existing
Runge--Kutta methods.

In future work, we will embed our approach in a high level scheme, based
on a branching algorithm to also verify properties such as stability or
symplecticity, with the same verification procedures as are presented in
this paper.

\section*{Acknowledgements}
 I am grateful to John Butcher for his comments on the original version of this paper. His remarks led to this extended version.
%I would like to show my gratitude to John for his very interesting comments that greatly enriched this version of the paper. I am also very proud that John took time to read my article and send me his remarks. 

\bibliographystyle{plain}
\bibliography{rk_interval}

\begin{thebibliography}{10}

\bibitem{alexandre2016validated}
Julien Alexandre~dit Sandretto and Alexandre Chapoutot.
\newblock {V}alidated {E}xplicit and {I}mplicit {R}unge-{K}utta {M}ethods.
\newblock {\em Reliable Computing}, 22:79--103, 2016.

\bibitem{Benhamou:1994:CR:864422}
Fr\'ed\'eric Benhamou, David McAllester, and Pascal Van~Hentenryck.
\newblock {CLP} ({I}ntervals) {R}evisited.
\newblock Technical report, Brown University, Providence, RI, USA, 1994.

\bibitem{Bornemann01}
Folkmar Bornemann.
\newblock {R}unge-{K}utta {M}ethods, {T}rees, and {M}aple - {O}n a {S}imple
  {P}roof of {B}utcher's {T}heorem and the {A}utomatic {G}eneration of {O}rder
  {C}ondition.
\newblock {\em Selcuk Journal of Applied Mathematics}, 2(1), 2001.

\bibitem{Butcher63}
John~C. Butcher.
\newblock Coefficients for the {S}tudy of {Runge-Kutta} {I}ntegration
  {P}rocesses.
\newblock {\em Journal of the Australian Mathematical Society}, 3:185--201,
  1963.

\bibitem{butcher2003numerical}
John~C. Butcher.
\newblock {\em {N}umerical {M}ethods for {O}rdinary {D}ifferential
  {E}quations}.
\newblock Wiley, 2003.

\bibitem{chabert2009contractor}
Gilles Chabert and Luc Jaulin.
\newblock Contractor {P}rogramming.
\newblock {\em {A}rtificial {I}ntelligence}, 173(11):1079--1100, 2009.

\bibitem{euler_1792}
Leonhard Euler.
\newblock {\em {I}nstitutiones {C}alculi {I}ntegralis}.
\newblock Academia Imperialis Scientiarum, 1792.

\bibitem{rk14}
Terry Feagin.
\newblock {H}igh-order {E}xplicit {R}unge-{K}utta {M}ethods {U}sing
  {M}-{S}ymmetry.
\newblock {\em Neural, Parallel \& Scientific Computations}, 20(4):437--458,
  2012.

\bibitem{Griewank:SIAM00}
Andreas Griewank.
\newblock {\em {Evaluating Derivatives: Principles and Techniques of
  Algorithmic Differentiation}}.
\newblock {Society for Industrial and Applied Mathematics}, Philadelphia, PA,
  USA, 2000.

\bibitem{HNW93}
Ernst Hairer, Syvert~P. N\o{}rsett, and Gerhard Wanner.
\newblock {\em Solving Ordinary Differential Equations~{I}: Nonstiff Problems}.
\newblock Springer-Verlag, 2nd edition, 2009.

\bibitem{hansen_global_2003}
Eldon~R. Hansen.
\newblock {\em Global Optimization Using Interval Analysis}.
\newblock Marcel Dekker Inc., 2003.

\bibitem{JKDW01}
Luc Jaulin, Michel Kieffer, Olivier Didrit, and Eric Walter.
\newblock {\em {Applied Interval Analysis}}.
\newblock Springer, 2001.

\bibitem{kutta1901beitrag}
Martin~W. Kutta.
\newblock Beitrag zur {N}\"aherungsweisen {I}ntegration {T}otaler
  {D}ifferentialgleichungen.
\newblock {\em Zeit. Math. Phys.}, 46:435--53, 1901.

\bibitem{Lebbah2002109}
Yahia Lebbah and Olivier Lhomme.
\newblock {A}ccelerating {F}iltering {T}echniques for {N}umeric {CSPs}.
\newblock {\em Artificial Intelligence}, 139(1):109--132, 2002.

\bibitem{Lhomme93consistencytechniques}
Olivier Lhomme.
\newblock {C}onsistency {T}echniques for {N}umeric {CSPs}.
\newblock In {\em Proceedings of the 13th International Joint Conference on
  Artifical Intelligence}, volume~1, pages 232--238, 1993.

\bibitem{marciniak2004representation}
Andrzej Marciniak and Barbara Szyszka.
\newblock {O}n {R}epresentation of {C}oefficients in {I}mplicit {I}nterval
  {M}ethods of {R}unge-{K}utta {T}ype.
\newblock {\em Computational Methods in Science and Technology}, 10(1):57--71,
  2004.

\bibitem{radau17}
Jesus Mart\'in-Vaquero.
\newblock {A} 17th-order {R}adau {IIA} {M}ethod for {P}ackage {RADAU}.
\newblock {\em Applications in mechanical systems, Computers \& Mathematics
  with Applications}, 2010.

\bibitem{moore66}
Ramon Moore.
\newblock {\em {I}nterval {A}nalysis}.
\newblock Prentice Hall, 1966.

\bibitem{handbookfloat09}
Jean-Michel {M}uller, Nicolas {B}risebarre, Florent {D}e {D}inechin,
  Claude-Pierre {J}eannerod, Vincent {L}ef{\`e}vre, Guillaume {M}elquiond,
  Nathalie {R}evol, Damien {S}tehl{\'e}, and Serge {T}orres.
\newblock {\em {H}andbook of {F}loating-{P}oint {A}rithmetic}.
\newblock {B}irkhauser, 2009.

\bibitem{ralston}
Anthony Ralston.
\newblock {R}unge-{K}utta {M}ethods with {M}inimum {E}rror {B}ounds.
\newblock {\em Mathematics of computation}, pages 431--437, 1962.

\bibitem{rueher_csp_2005}
Michel Rueher.
\newblock {S}olving {C}ontinuous {C}onstraint {S}ystems.
\newblock In {\em Proc. of 8th International Conference on Computer Graphics
  and Artificial Intelligence (3IA'2005)}, 2005.

\bibitem{ibex_website}
Ibex Team.
\newblock Ibex.
\newblock \url{http://ibex-lib.org/}.

\end{thebibliography}

\end{document}